\documentclass[12pt]{amsart}
\usepackage{amsmath,amssymb,amsfonts,amsthm}
\usepackage[all,cmtip]{xy}
\usepackage{fullpage}
\begin{document}
\theoremstyle{plain}
\newtheorem{thm}{Theorem}[subsection]
\newtheorem{lem}[thm]{Lemma}
\newtheorem{cor}[thm]{Corollary}
\newtheorem{prop}[thm]{Proposition}
\newtheorem{rem}[thm]{Remark}
\newtheorem{defn}[thm]{Definition}
\newtheorem{ex}[thm]{Example}
\newtheorem{ques}[thm]{Question}
\newtheorem{fact}[thm]{Fact}
\newtheorem{conj}[thm]{Conjecture}
\numberwithin{equation}{subsection}
\def\theequation{\thesection.\arabic{equation}}
\newcommand{\mc}{\mathcal}
\newcommand{\mb}{\mathbb}
\newcommand{\surj}{\twoheadrightarrow}
\newcommand{\inj}{\hookrightarrow}
\newcommand{\zar}{{\rm zar}}
\newcommand{\an}{{\rm an}} 
\newcommand{\red}{{\rm red}}
\newcommand{\codim}{{\rm codim}}
\newcommand{\rank}{{\rm rank}}
\newcommand{\Ker}{{\rm Ker \,}}
\newcommand{\Pic}{{\rm Pic}}
\newcommand{\Div}{{\rm Div}}
\newcommand{\Hom}{{\rm Hom}}
\newcommand{\im}{{\rm im \,}}
\newcommand{\Spec}{{\rm Spec \,}}
\newcommand{\Sing}{{\rm Sing}}
\newcommand{\Char}{{\rm char}}
\newcommand{\Tr}{{\rm Tr}}
\newcommand{\Gal}{{\rm Gal}}
\newcommand{\Min}{{\rm Min \ }}
\newcommand{\Max}{{\rm Max \ }}
\newcommand{\CH}{{\rm CH}}
\newcommand{\pr}{{\rm pr}}
\newcommand{\cl}{{\rm cl}}
\newcommand{\gr}{{\rm Gr }}
\newcommand{\Coker}{{\rm Coker \,}}
\newcommand{\id}{{\rm id}}
\newcommand{\Rep}{{\bold {Rep} \,}}
\newcommand{\Aut}{{\rm Aut}}
\newcommand{\GL}{{\rm GL}}
\newcommand{\Bl}{{\rm Bl}}
\newcommand{\Jab}{{\rm Jab}}
\newcommand{\alb}{\rm Alb}
\newcommand{\NS}{\rm NS}
\newcommand{\sA}{{\mathcal A}}
\newcommand{\sB}{{\mathcal B}}
\newcommand{\sC}{{\mathcal C}}
\newcommand{\sD}{{\mathcal D}}
\newcommand{\sE}{{\mathcal E}}
\newcommand{\sF}{{\mathcal F}}
\newcommand{\sG}{{\mathcal G}}
\newcommand{\sH}{{\mathcal H}}
\newcommand{\sI}{{\mathcal I}}
\newcommand{\sJ}{{\mathcal J}}
\newcommand{\sK}{{\mathcal K}}
\newcommand{\sL}{{\mathcal L}}
\newcommand{\sM}{{\mathcal M}}
\newcommand{\sN}{{\mathcal N}}
\newcommand{\sO}{{\mathcal O}}
\newcommand{\sP}{{\mathcal P}}
\newcommand{\sQ}{{\mathcal Q}}
\newcommand{\sR}{{\mathcal R}}
\newcommand{\sS}{{\mathcal S}}
\newcommand{\sT}{{\mathcal T}}
\newcommand{\sU}{{\mathcal U}}
\newcommand{\sV}{{\mathcal V}}
\newcommand{\sW}{{\mathcal W}}
\newcommand{\sX}{{\mathcal X}}
\newcommand{\sY}{{\mathcal Y}}
\newcommand{\sZ}{{\mathcal Z}}
\newcommand{\A}{{\mathbb A}}
\newcommand{\B}{{\mathbb B}}
\newcommand{\C}{{\mathbb C}}
\newcommand{\D}{{\mathbb D}}
\newcommand{\E}{{\mathbb E}}
\newcommand{\F}{{\mathbb F}}
\newcommand{\G}{{\mathbb G}}
\renewcommand{\H}{{\mathbb H}}
\newcommand{\I}{{\mathbb I}}
\newcommand{\J}{{\mathbb J}}
\newcommand{\M}{{\mathbb M}}
\newcommand{\N}{{\mathbb N}}
\renewcommand{\P}{{\mathbb P}}
\newcommand{\Q}{{\mathbb Q}}
\newcommand{\R}{{\mathbb R}}
\newcommand{\T}{{\mathbb T}}
\newcommand{\V}{{\mathbb V}}
\newcommand{\W}{{\mathbb W}}
\newcommand{\X}{{\mathbb X}}
\newcommand{\Y}{{\mathbb Y}}
\newcommand{\Z}{{\mathbb Z}}
\newcommand{\Nwt}{{\rm Nwt}}
\newcommand{\Hdg}{{\rm Hdg}}
\newcommand{\ind}{{\rm ind \,}}
\newcommand{\Br}{{\rm Br}}
\newcommand{\inv}{{\rm inv}}
\newcommand{\Nm}{{\rm Nm}}
\newcommand{\Griff}{{\rm Griff}}
\newcommand{\Image}{\rm Im \,}
\newcommand{\Ev}{\rm Ev \,}
\title[Unramified cohomology]{On vanishing of unramified cohomology of geometrically rational varieties over finite fields}
\author{{Nguyen Le Dang Thi}}
\email{ le.nguyen@uni-due.de}
\date{31. October 2011}          
\begin{abstract}
The purpose of this note is to show that the third unramified cohomology $H^0_{Zar}(X,\sH^3(\Q_{\ell}/\Z_{\ell}(2)))$ of a smooth projective geometrically rational variety $X$ of dimension $3$ over a finite field $k=\F_q$ must vanish under $\Z_{\ell}$-exactness Hard Lefschetz condition. 
\end{abstract}
\maketitle
\section{Introduction}
Let $k$ be a field and $\ell \neq char(k) = p$ be any prime. Let $X$ be a smooth projective geometrically integral $k$-variety. Denote by $\sH^n_{\acute{e}t}(\Q_{\ell}/\Z_{\ell}(j))$ resp. $\sH^n_{\acute{e}t}(\mu_{\ell^m}^{\otimes j})$ the Zariski sheaf on $X$ associated to the presheaf $U \mapsto H^n_{\acute{e}t}(U,\Q_{\ell}/\Z_{\ell}(j))$ resp. $U \mapsto H^n_{\acute{e}t}(U,\mu_{\ell^m}^{\otimes j})$. If $F=k(X)$ is the function field of $X$, then we write $H^n_{nr}(F/k,\Q_{\ell}/\Z_{\ell}(j))$ resp. $H^n_{nr}(F/k,\mu_{\ell^m}^{\otimes j})$ for the unramified cohomology with $\ell$-divisible resp. finite coefficients. We denote by $\alpha: X_{\acute{e}t} \rightarrow X_{Zar}$ the obvious morphism. In fact, one has $\sH^n_{\acute{e}t}(A(j)) = R^n\alpha_*A(j)$. By a geometrically rational variety over a field $k$ we mean a smooth projective variety $X$ such that $\overline{X}=X\otimes_k\overline{k}$ is a rational variety. For a smooth projective variety $Y$ of dimension $n+1$ over $\overline{\F}_q$ with a smooth hyperplane section $Z$, we say that $Y$ satisfies $\Z_{\ell}$-exactness Hard Lefschetz condition if one has a direct decomposition 
\begin{equation}\label{eqHL}
H^n_{\acute{e}t}(Z,\Z_{\ell}) = H^n_{\acute{e}t}(Z,\Z_{\ell})_{ev} \oplus H^n_{\acute{e}t}(Y,\Z_{\ell})
\end{equation}
where we choose an isomorphism $\Z_{\ell} \simeq \Z_{\ell}(1)$ and forget about Tate-twist and $H^n_{\acute{e}t}(Z,\Z_{\ell})_{ev}$ denotes the space of vanishing cycles of $H^n_{\acute{e}t}(Z,\Z_{\ell})$.  
Our main result is the following theorem: 
\begin{thm}\label{mainthm}
Let $X$ be a smooth projective geometrically rational variety of dimension $3$ over a finite field $k=\F_q$ with function field $F=k(X)$ and $\ell \neq char(k)=p$ be a prime such that $\overline{X}$ satisfies the $\Z_{\ell}$-exactness Hard Lefschetz condition, then the third unramified cohomology $H^3_{nr}(F/k,\Q_{\ell}/\Z_{\ell}(2))$ is trivial. 
\end{thm}
The $\Z_{\ell}$-exactness Hard Lefschetz condition is in fact the question in \cite[Ques. 5.7]{CTK11}, which we certainly can not answer in this note. 
\section{Proof of theorem \ref{mainthm}}
In this section we prove the main theorem \ref{mainthm} through several steps. First of all we show 
\begin{prop}\label{prop1} 
Let $X$ be a smooth projective geometrically integral variety over a field $k$ of characteristic $char(k) \geq 0$ and $\ell \neq char(k)$ be a prime. Then one has an exact sequence 
\begin{equation}\label{eq2}
0 \rightarrow \CH^1(X)\otimes \Z_{\ell} \rightarrow H^2_{\acute{e}t}(X,\Z_{\ell}(2)) \rightarrow \underset{n}{\varprojlim}H^2_{nr}(F/k,\mu_{\ell^n}) 
\end{equation}
Moreover, the group $\underset{n}{\varprojlim}H^2_{nr}(F/k,\mu_{\ell^n})$ is torsion-free.  
\end{prop}
\begin{proof}
By Kummer theory one has a distinguished triangle, see \cite[Thm. 6.6]{Voe03}
\begin{equation*}
\Z/\ell^n(1) \rightarrow R\alpha_*\alpha^*\Z/\ell^n(1) \rightarrow \tau_{\geq 2} R\alpha_*\alpha^*\Z/\ell^n(1) \stackrel{+1}{\longrightarrow}
\end{equation*}
By taking cohomology we have an exact sequence 
\begin{equation*}
0 \rightarrow \CH^1(X)\otimes \Z/\ell^n \rightarrow H^2_{\acute{e}t}(X,\mu_{\ell^n}) \rightarrow \H^2_{Zar}(X,\tau_{\geq 2}R\alpha_*\alpha^*\Z/\ell^n(1)) \rightarrow 0
\end{equation*}
One has a spectral sequence \cite[Thm. 0.3]{SV00}
\begin{equation*}\label{eq3}
E^{p,q}_2 = H^p_{Zar}(X,\underline{H}^q(\tau_{\geq 2}R\alpha_*\alpha^*\Z/\ell^n(1))) \Rightarrow \H^{p+q}_{Zar}(X,\tau_{\geq 2}R\alpha_*\alpha^*\Z/\ell^n(1)), 
\end{equation*}
where $\underline{H}^q$ denote the cohomology sheaves. By the exact sequence for terms of lower degree one has an injection 
\begin{equation*}
0 \rightarrow \H^2_{Zar}(X,\tau_{\geq 2 }R\alpha_*\alpha^*\Z/\ell^n(1)) \rightarrow H^2_{nr}(F/k,\mu_{\ell^n})
\end{equation*}
Since there is no differentials for $E^{0,2}_r$, we have $H^2_{nr}(F/k,\mu_{\ell^n}) \stackrel{def}{=} E^{0,2}_2 = E^{0,2}_{\infty}$. So the injection above is in fact an isomorphism, so it gives us the exact sequence \ref{eq2}. Now by definition we have $\underset{n}{\varprojlim}H^2_{nr}(F/k,\mu_{\ell^n}) \subset H^2_{\acute{e}t}(F,\Z_{\ell}(1))$. The last group is torsion-free by Kummer theory, so we are done.  
\end{proof}
\begin{prop}\label{prop4}
Let $X$ be a smooth projective geometrically integral variety of dimension $d$ over a field $k$ with the function field $F=k(X)$. Let $k \subset \Omega$ be a universal domain in sense of Weil. Assume $\CH_0(X_{\Omega}) = \Z$, then $H^p_{nr}(F/k,\mu_{\ell^n}^{\otimes j})$ are killed by an integer $N \geq 1$, for all $p>cd_{\ell}(k)$.  
\end{prop}
\begin{proof}
The assumption that $\CH_0(X_{\Omega}) = \Z$ implies the diagonal decomposition in $\CH^d(X\times X)$ (see \cite{BS83})
\begin{equation*}
N\Delta_X = \Gamma_1 + \Gamma_2,
\end{equation*}
where $\Gamma_1$ is supported on $\xi \times X$ with $\xi$ is a $0$-dimensional subscheme, $\Gamma_2$ is supported on $X \times D$ for a divisor $D\subset X$ and $N \in \N^{\times}$ is an integer. By action of correspondences, see e.g. \cite[App.]{CTV10}, we obtain 
\begin{equation*}
NId = \Gamma_{1_*} + \Gamma_{2_*}: H^0_{Zar}(X,\sH^p_{\acute{e}t}(\mu_{\ell^n}^{\otimes j})) \rightarrow H^0_{Zar}(X,\sH^p_{\acute{e}t}(\mu_{\ell^n}^{\otimes j})).
\end{equation*}
One has that $\Gamma_{1_*}$ factors through 
\begin{equation*}
H^0_{Zar}(X,\sH^p_{\acute{e}t}(\mu_{\ell^n}^{\otimes j})) \rightarrow H^0_{Zar}(\xi,\sH^p_{\acute{e}t}(\mu_{\ell^n}^{\otimes j})), 
\end{equation*} 
where we can assume $\xi$ is a closed point and so $H^0_{Zar}(\xi,\sH^p_{\acute{e}t}(\mu_{\ell^n}^{\otimes j}))$ is trivial for $p > cd_{\ell}(k)$. One has that $\Gamma_{2_*} = 0$, since $\Gamma_{2_*}$ is supported on $D \subsetneq X$. 
This shows that $H^0_{Zar}(X,\sH^p_{\acute{e}t}(\mu_{\ell^n}^{\otimes j}))$ are killed by an integer $N \geq 1$ for all $p>cd_{\ell}(k)$.
\end{proof}
\begin{prop}\label{prop10} 
Let $X$ be a smooth projective geometrically rational variety of dimension $3$ over a finite field $\F_q$, then $H^3_{\acute{e}t}(X,\Q_{\ell}/\Z_{\ell}(2)) = 0$. In particular, $H^4_{\acute{e}t}(X,\Z_{\ell}(2))$ is torsion-free. 
\end{prop}
\begin{proof}
Let $G = \Gal(\overline{\F_q}/\F_q)$ be the absolute Galois group of $\F_q$. From the Hochschild-Serre spectral sequence 
\begin{equation*}
E^{a,b}_2 = H^a_{Gal}(\F_q,H^b_{\acute{e}t}(\overline{X},\Q_{\ell}/\Z_{\ell}(2))) \Rightarrow H^{a+b}_{\acute{e}t}(X,\Q_{\ell}/\Z_{\ell}(2))
\end{equation*}
one has a short exact sequence 
\begin{equation*}
0 \rightarrow H^1_{Gal}(\F_q,H^2_{\acute{e}t}(\overline{X},\Q_{\ell}/\Z_{\ell}(2))) \rightarrow H^3_{\acute{e}t}(X,\Q_{\ell}/\Z_{\ell}(2)) \rightarrow H^3_{\acute{e}t}(\overline{X},\Q_{\ell}/\Z_{\ell}(2))^G \rightarrow 0
\end{equation*}
From the universal coefficient exact sequence 
\begin{equation}\label{equc}
0 \rightarrow H^2_{\acute{e}t}(\overline{X},\Z_{\ell}(2))\otimes \Q_{\ell}/\Z_{\ell} \rightarrow H^2_{\acute{e}t}(\overline{X},\Q_{\ell}/\Z_{\ell}(2)) \rightarrow H^3_{\acute{e}t}(\overline{X},\Z_{\ell}(2))_{tors} \rightarrow 0
\end{equation} 
and from the fact by Serre, see e.g. \cite{A-M}, that $H^3_{\acute{e}t}(\overline{X},\Z_{\ell}(2))$ is torsion-free, we see that $H^2_{\acute{e}t}(\overline{X},\Q_{\ell}/\Z_{\ell}(2))$ is divisible. Since $\F_q$ has cohomological dimension $1$, it implies $H^1_{Gal}(\F_q,H^2_{\acute{e}t}(\overline{X},\Q_{\ell}/\Z_{\ell}(2)))$ is also divisible. By Weil conjecture, see e.g. \cite{CTSS83}, the group $H^3_{\acute{e}t}(X,\Q_{\ell}/\Z_{\ell}(2))$ is finite, so $H^1_{Gal}(\F_q,H^2_{\acute{e}t}(\overline{X},\Q_{\ell}/\Z_{\ell}(2)))$ is trivial and we must have 
\begin{equation*}
H^3_{\acute{e}t}(X,\Q_{\ell}/\Z_{\ell}(2)) \cong H^3_{\acute{e}t}(\overline{X},\Q_{\ell}/\Z_{\ell}(2))^G
\end{equation*}  
So it is enough to show that $H^3_{\acute{e}t}(\overline{X},\Q_{\ell}/\Z_{\ell}(2)) = 0$. By universal coefficient exact sequence 
\begin{equation*}
0 \rightarrow H^3_{\acute{e}t}(\overline{X},\Z_{\ell}(2)) \otimes \Q_{\ell}/\Z_{\ell} \rightarrow H^3_{\acute{e}t}(\overline{X},\Q_{\ell}/\Z_{\ell}(2)) \rightarrow H^4_{\acute{e}t}(\overline{X},\Z_{\ell}(2))_{tors} \rightarrow 0, 
\end{equation*}
and the fact by Serre that $H^4_{\acute{e}t}(\overline{X},\Z_{\ell}(2))$ is torsion-free, we conclude that 
$$H^3_{\acute{e}t}(\overline{X},\Q_{\ell}/\Z_{\ell}(2)) = H^3_{\acute{e}t}(\overline{X},\Z_{\ell}(2)) \otimes \Q_{\ell}/\Z_{\ell}.$$
The last group is by \cite[Cor. 4.20]{Kah11} isomorphic to the torsion subgroup of the kernel of the map 
$$\H^4_{\acute{e}t}(\overline{X},\Z(2))\otimes \Z_{\ell} \rightarrow H^4_{\acute{e}t}(\overline{X},\Z_{\ell}(2)),$$
where we denote by $\H^n_{\acute{e}t}(-,\Z(j))$ the \'etale motivic cohomology. Consider the cycle class map  
\begin{equation}\label{eqcycl}
cl^2_{\overline{X}}: \CH^2(\overline{X})\otimes \Z_{\ell} \rightarrow \H^4_{\acute{e}t}(\overline{X},\Z(2))\otimes \Z_{\ell} \rightarrow H^4_{\acute{e}t}(\overline{X},\Z_{\ell}(2))
\end{equation} 
From the Bloch-Ogus spectral sequence \cite{BO74} 
\begin{equation*}
E^{i,j}_2 = H^i_{Zar}(\overline{X},\sH^j_{\acute{e}t}(\Z_{\ell}(2))) \Rightarrow H^{i+j}_{\acute{e}t}(\overline{X},\Z_{\ell}(2))
\end{equation*}
one has an exact sequence 
\begin{multline*}
0 \rightarrow N^1H^3_{\acute{e}t}(\overline{X},\Z_{\ell}(2)) \rightarrow H^3_{\acute{e}t}(\overline{X},\Z_{\ell}(2)) \rightarrow H^0_{Zar}(\overline{X},\sH^3_{\acute{e}t}(\Z_{\ell}(2))) \rightarrow \\ \rightarrow \CH^2(\overline{X})\otimes \Z_{\ell} \stackrel{cl^2_{\overline{X}}}{\longrightarrow} H^4_{\acute{e}t}(\overline{X},\Z_{\ell}(2)), 
\end{multline*}
where $N^1$ is the first step coniveau filtration. Since $H^0_{Zar}(\overline{X},\sH^3_{\acute{e}t}(\Z_{\ell}(2))) = 0$, we get the injectivity of $cl^2_{\overline{X}}$. Moreover, from \cite[Prop. 2.8]{Kah11} one has an exact sequence 
\begin{equation*}
0 \rightarrow \CH^2(\overline{X})\otimes \Z_{\ell} \rightarrow \H^4_{\acute{e}t}(\overline{X},\Z(2))\otimes \Z_{\ell} \rightarrow H^0_{Zar}(\overline{X},\sH^3_{\acute{e}t}(\Q_{\ell}/\Z_{\ell}(2))) \rightarrow 0
\end{equation*}
Since $H^0_{Zar}(\overline{X},\sH^3_{\acute{e}t}(\Q_{\ell}/\Z_{\ell}(2))) = 0$, we have an isomorphism 
$$\CH^2(\overline{X})\otimes \Z_{\ell} \cong \H^4_{\acute{e}t}(\overline{X},\Z(2))\otimes \Z_{\ell}.$$  
Apply now the Kernel-Cokernel exact sequence for the composition \ref{eqcycl}, we can conclude that $\H^4_{\acute{e}t}(\overline{X},\Z(2))\otimes \Z_{\ell}$ maps injectively to $H^4_{\acute{e}t}(\overline{X},\Z_{\ell}(2))$. So $H^3_{\acute{e}t}(\overline{X},\Q_{\ell}/\Z_{\ell}(2))$ is trivial. Now from the exact sequence 
\begin{equation*}
\cdots \rightarrow H^3_{\acute{e}t}(X,\Q_{\ell}/\Z_{\ell}(2)) \rightarrow H^4_{\acute{e}t}(X,\Z_{\ell}(2)) \rightarrow H^4_{\acute{e}t}(X,\Q_{\ell}) \rightarrow \cdots 
\end{equation*}
we see that $H^4_{\acute{e}t}(X,\Z_{\ell}(2))$ is torsion-free. 
\end{proof}
\begin{prop}\label{prop5}
Let $X$ be a smooth projective geometrically rational variety of dimension $3$ over a finite field $\F_q$. Assume that $\overline{X}$ satisfies the condition \ref{eqHL}, then the cycle class map 
\begin{equation*}
cl_{X}^2: \CH^2(X)\otimes \Z_{\ell} \rightarrow H^4_{\acute{e}t}(X,\Z_{\ell}(2))
\end{equation*}
is surjective. 
\end{prop}
\begin{proof}
$X$ is geometrically rational, we have the base change condition $\CH_0(X_{\Omega}) = \Z$. So by \ref{prop1} and \ref{prop4}, we have a surjection $cl_X^1: \CH^1(X)\otimes \Z_{\ell} \surj H^2(X,\Z_{\ell}(1))$. Let $H$ be a smooth hyperplane section (over $\F_q$ see \cite{Poo04}) and $G=\Gal(\overline{\F_q}/\F_q)$. Consider the commutative diagram 
\begin{equation}\label{eq6}
\xymatrix{ \CH^1(X)\otimes \Z_{\ell} \ar[d]^{-\cap H} \ar[r]^{\cong} & H^2_{\acute{e}t}(X,\Z_{\ell}(1)) \ar[d]^{-\cap H} \ar@^{->>}[r] & H^2_{\acute{e}t}(\overline{X},\Z_{\ell}(1))^G \ar[d]^{-\cap \overline{H}} \\ \CH^2(X)\otimes \Z_{\ell} \ar[r]^{cl_X^2} & H^4_{\acute{e}t}(X,\Z_{\ell}(2)) \ar@^{->>}[r] & H^4_{\acute{e}t}(\overline{X},\Z_{\ell}(2))^G}
\end{equation}
Since $H^2_{\acute{e}t}(\overline{X},\Z_{\ell}(1))$ and $H^4_{\acute{e}t}(\overline{X},\Z_{\ell}(2))$ are torsion-free by Serre, see e.g. \cite{A-M}, the $G$-equivariant map 
\begin{equation*}
-\cap \overline{H}: H^2_{\acute{e}t}(\overline{X},\Z_{\ell}(1)) \rightarrow H^4_{\acute{e}t}(\overline{X},\Z_{\ell}(2))
\end{equation*}
is then an isomorphism under our assumption \ref{eqHL} by Hard Lefschetz theorem \cite[Thm. 4.1.1]{Del80} (see \cite[p. 223]{Del80} for $\Z_{\ell}$-cohomology). From the commutative diagram \ref{eq6} we can conclude that $\CH^2(X)\otimes \Z_{\ell} \rightarrow H^4_{\acute{e}t}(\overline{X},\Z_{\ell}(2))^G$ is surjective. Over a finite field $\F_q$, the Hochschild-Serre spectral sequence 
\begin{equation*}
E^{a,b}_2 = H^a_{Gal}(\F_q,H^b_{\acute{e}t}(\overline{X},\Z_{\ell}(2))) \Rightarrow H^{a+b}_{\acute{e}t}(X,\Z_{\ell}(2))
\end{equation*}
breaks up into short exact sequence 
\begin{equation*}
0 \rightarrow H^1_{Gal}(\F_q,H^3_{\acute{e}t}(\overline{X},\Z_{\ell}(2))) \rightarrow H^4_{\acute{e}t}(X,\Z_{\ell}(2)) \rightarrow H^4_{\acute{e}t}(\overline{X},\Z_{\ell}(2))^G \rightarrow 0.
\end{equation*}
Apply now the Kernel-Cokernel exact sequence for the bottom maps of the commutative diagram \ref{eq6} above, we have an exact sequence 
\begin{multline*}
0 \rightarrow \Ker(cl^2_X) \rightarrow  \Ker (\CH^2(X)\otimes \Z_{\ell} \rightarrow H^4_{\acute{e}t}(\overline{X},\Z_{\ell}(2))^G) \rightarrow \\ \rightarrow H^1_{Gal}(\F_q,H^3_{\acute{e}t}(\overline{X},\Z_{\ell}(2))) \rightarrow \Coker(cl^2_X) \rightarrow 0
\end{multline*}
By Weil conjecture, see e.g. \cite{CTSS83}, $H^1_{Gal}(\F_q,H^3_{\acute{e}t}(\overline{X},\Z_{\ell}(2)))$ is finite, but from \ref{prop10} we have $H^4_{\acute{e}t}(X,\Z_{\ell}(2))$ is torsion-free, so $H^1_{Gal}(\F_q,H^3_{\acute{e}t}(\overline{X},\Z_{\ell}(2)))$ must vanish, hence $\Coker(cl^2_X) = 0$. 
\end{proof}
\begin{rem}\label{rem7}{\rm
In fact, the cycle class map $cl^2_X$ is an isomorphism for $X$ a smooth projective geometrically rational threefold over a finite field $\F_q$ under condition \ref{eqHL}. The surjectivity is proved above in \ref{prop5} under the condition \ref{eqHL}. The injectivity follows only from the fact that $H^0_{Zar}(X,\sH^3_{\acute{e}t}(\Z_{\ell}(2)))$ is torsion-free by Merkurjev-Suslin theorem ($H^3_{\acute{e}t}(F,\Z_{\ell}(2))$ is torsion-free) hence it must vanish by \ref{prop4} and from the exact sequence of Bloch-Ogus spectral sequence \cite{BO74} 
\begin{multline*}
0 \rightarrow N^1H^3_{\acute{e}t}(X,\Z_{\ell}(2)) \rightarrow H^3_{\acute{e}t}(X,\Z_{\ell}(2)) \rightarrow H^0_{Zar}(X,\sH^3_{\acute{e}t}(\Z_{\ell}(2))) \rightarrow \\ \rightarrow \CH^2(X)\otimes \Z_{\ell} \rightarrow H^4_{\acute{e}t}(X,\Z_{\ell}(2)),
\end{multline*}
without condition \ref{eqHL}.
}
\end{rem}
Now we use the following theorem of B. Kahn
\begin{thm}\cite[Thm. 1.1]{Kah11}\label{thm8}
Let $X$ be a smooth projective variety over a field $k$ and $\ell \neq char(k)$ be a prime. One has an exact sequence 
\begin{equation}\label{eq9}
0 \rightarrow H^0_{Zar}(X,\sH^3_{\acute{e}t}(\Z_{\ell}(2)))\otimes \Q/\Z \rightarrow H^0_{Zar}(X,\sH^3_{\acute{e}t}(\Q_{\ell}/\Z_{\ell}(2))) \rightarrow C_{tors} \rightarrow 0,
\end{equation} 
where $C_{tors}$ is the torsion subgroup of the cokernel of $cl^2_X$.
\end{thm}
As $C_{tors} = 0$ by \ref{prop5}, so $H^3_{nr}(F/k,\Q_{\ell}/\Z_{\ell}(2))$ is divisible by \ref{eq9}, so it must vanish by \ref{prop4}, so we finish the proof of the theorem \ref{mainthm}. 
\begin{rem}\label{remHP}
\rm{Let $X$ be a smooth projective threefold over an algebraic closure $\overline{\F}$ of a finite field $\F_q$ with a smooth ample divisor $Y \inj X$. If the Brauer group $\Br(Y)$ is finite, then $\CH^2(X)\otimes \Z_{\ell}$ maps surjectively onto $H^4_{\acute{e}t}(X,\Z_{\ell}(2))$. Indeed, $\underset{n}{\varprojlim}H^2_{nr}(\overline{\F}(Y)/\overline{\F},\mu_{\ell^n})$ will be trivial under the assumption of finiteness of $\Br(Y)$. So by \ref{prop1}, we have $\CH^1(Y) \surj H^2_{\acute{e}t}(Y,\Z_{\ell}(1))$. By weak Lefschetz theorem \cite{Del80} one has a surjection $H^2_{\acute{e}t}(Y,\Z_{\ell}(1)) \surj H^4_{\acute{e}t}(X,\Z_{\ell}(2))$. So $H^4_{\acute{e}t}(X,\Z_{\ell}(2))$ is generated by $1$-cycles as one looks at the following commutative diagram 
\begin{equation*}
\xymatrix{  \CH^1(Y)\otimes \Z_{\ell} \ar@^{->>}[d] \ar[r] & \CH^2(X) \otimes \Z_{\ell} \ar[d] \ar[r] & \CH^2(X-Y) \otimes \Z_{\ell} \ar[r] & 0 \\ H^2_{\acute{e}t}(Y,\Z_{\ell}(1)) \ar@^{->>}[r] & H^4_{\acute{e}t}(X,\Z_{\ell}(2))  }
\end{equation*}
}
\end{rem}

\thanks{$\bold{Acknowledgement:}$ I thank J.-L. Colliot-Th\'el\`ene and H. Esnault for reading an earlier version of this note, which contains a mistake in the proof of theorem \ref{mainthm}.}   
\bibliographystyle{plain}
\renewcommand\refname{References}

\end{document}